\documentclass[a4paper,reqno,11pt]{amsart}
\usepackage{t1enc}
\usepackage[margin=1.0in]{geometry}
\usepackage{ifthen}
\usepackage{graphicx}
\usepackage{xcolor}

\newcommand{\R}{\mathbf{R}}

\newcommand{\pr}{\mathbf{P}}
\newcommand{\ex}{\mathbf{E}}

\theoremstyle{plain}
\newtheorem{theorem}{Theorem}
\newtheorem{lemma}{Lemma}

\newtheorem{proposition}{Proposition}

\theoremstyle{definition}

\newtheorem{remark}{Remark}

\theoremstyle{remark}

\newcommand{\formula}[2][nolabel]
{\ifthenelse{\equal{#1}{nolabel}}
 {\begin{align*} #2 \end{align*}}
 {\ifthenelse{\equal{#1}{}}
  {\begin{align} #2 \end{align}}
  {\begin{align} \label{#1} #2 \end{align}}
 }
}

\numberwithin{equation}{section}

\begin{document}

%

\title[Heat kernel estimates for the Bessel differential operator in half-line]{Heat kernel estimates for the Bessel differential operator in half-line}
\author{Kamil Bogus, Jacek Ma{\l}ecki}
\address{Kamil Bogus, Jacek Ma{\l}ecki \\ Faculty of Fundamental Problems of Technology \\ Department of Mathematics\\ Wroc{\l}aw University of Technology \\ ul.
Wybrze{\.z}e Wyspia{\'n}\-skiego 27 \\ 50-370 Wroc{\l}aw,
Poland
}
\email{kamil.bogus@pwr.edu.pl, jacek.malecki@pwr.edu.pl}

\keywords{Bessel differential operator, heat kernel, sharp estimate, half-line, Bessel process, transition probability density}
\subjclass[2010]{35K08, 60J60}

\thanks{The project was funded by the National Science Centre grant no. 2013/11/D/ST1/02622.}

\begin{abstract} In the paper we consider the Bessel differential operator $L^{(\mu)}=\dfrac{d^2}{dx^2}+\dfrac{2\mu+1}{x}\dfrac{d}{dx}$ in half-line $(a,\infty)$, $a>0$, and its Dirichlet heat kernel $p_a^{(\mu)}(t,x,y)$. For $\mu=0$, by combining analytical and probabilistic methods, we provide sharp two-sided estimates of the heat kernel for the whole range of the space parameters $x,y>a$ and every $t>0$, which complements the recent results given in \cite{BogusMalecki:2014}, where the case $\mu\neq 0$ was considered.
\end{abstract}

\maketitle

\section{Introduction}
\label{section:introduction}
We consider the Bessel differential operator 
\formula{
  L^{(\mu)} = \dfrac{d^2}{dx^2}+\dfrac{2\mu+1}{x}\dfrac{d}{dx}\/,\quad \mu\in\R\/,
}
on the half-lines $(a,\infty)$, for every $a\geq 0$. Let us denote by $p_a^{(\mu)}(t,x,y)$ the corresponding Dirichlet heat kernel considered with respect to the reference measure $m^{(\mu)}(dy) = y^{2\mu+1}dy$. For $\mu=0$ we will simply write $p_a(t,x,y)$ and we will also omit superscripts in the notation of the other objects related to that case. The explicit formula for $p^{(\mu)}(t,x,y) := p^{(\mu)}_0(t,x,y)$ in terms of the modified Bessel functions is well-known
   \formula[eq:global:formula]
      {
p^{(\mu)}(t,x,y) = \frac1t (xy)^{-\mu}
\exp\left(-\frac{x^2+y^2}{2t}\right)I_{|\mu|}\left(\frac{xy}{t}\right)\/,\quad
x,y,t> 0\/.}
 For $a>0$ and $\mu\neq 0$ the sharp two-sided estimates of $p_a^{(\mu)}(t,x,y)$ of the form
 \formula[eq:munonzero:result]{
      p_a^{(\mu)}(t,x,y)\stackrel{\mu}{\approx} \left[1\wedge \frac{(x-a)(y-a)}{t}\right]\left(1\wedge \frac{xy}{t}\right)^{|\mu|-\frac{1}{2}} \frac{1}{(xy)^{\mu+1/2}}\frac{1}{\sqrt{t}}\exp\left(-\frac{(x-y)^2}{2t}\right)\/,
      }
for every $x,y>a$ and $t>0$, were obtained recently in \cite{BogusMalecki:2014}. Here $\stackrel{\mu}{\approx}$ means that the ratio of the functions on the left-hand side and the right-hand side is bounded from below and above by positive constants (depending only on $\mu$) for the indicated range of variables. Notice that the exponential behavior is here described explicitly, i.e. there are no different constants in the exponential terms in the lower and upper bounds. This is also the case with the main result of the present paper, where the two-sided sharp estimates of $p_1(t,x,y)$ (for $\mu=0$) for the whole range of the space and time parameters are provided in 
\begin{theorem}
\label{thm:main}
   We have
	\formula[eq:result:small]{
	   p_1(t,x,y)\approx\ln x\ln y\left(\ln\frac{3t}{x+\sqrt{t}}\ln\frac{3t}{y+\sqrt{t}}\right)^{-1}\frac{1}{t}\exp\left(-\frac{x^2+y^2}{2t}\right)\/,\quad xy\leq t\/,
	}
	 and 
	\formula[eq:result:large]{
	   p_1(t,x,y)\approx \left(1\wedge \frac{(x-1)(y-1)}{t}\right)\frac{1}{\sqrt{xyt}}\exp\left(-\frac{(x-y)^2}{2t}\right)\/,\quad xy>t\/,
	}
	for every $x,y>1$ and $t>0$.
\end{theorem}
 
Using the scaling property 
\formula{
  p_a(t,x,y) = \frac1a p_1(t/a^2,x/a,y/a)\/,\quad x,y>a\/,\quad t>0\/,
}
one can immediately deduce the corresponding result for any $a>0$. Moreover, the result can be rewritten in one formula covering both cases, i.e. $xy/t$ small and $xy/t$ large. More precisely, taking into account the asymptotic behavior of $I_0(z)$ at zero and infinity (see (\ref{eq:I:asym:infty})), we can write the estimate in the following way
\formula{
  \frac{p_a(t,x,y)}{p(t,x,y)}\approx 1\wedge \left[\ln (x/a)\ln (y/a) \left(\ln\frac{3xy+3t}{ax+a\sqrt{t}}\ln\frac{3xy+3t}{ay+a\sqrt{t}}\right)^{-1}\left(1+\frac{xy}{t}\right)\right]\/,
}
for every $x,y>a$ and $t>0$. Note that the constants appearing in the estimate are absolute and do not even depend on $a$. Notice that for $xy>t$ (and $a=1$) the estimates in (\ref{eq:munonzero:result}) are exactly of the same form as those given in (\ref{eq:result:large}). However, for $xy<t$ the logarithmic terms appear in (\ref{eq:result:small}), which is not the case in (\ref{eq:munonzero:result}). Moreover, the logarithmic behaviour of $p_a(t,x,y)$ require more delicate methods, i.e. those used in \cite{BogusMalecki:2014} can not be applied to prove (\ref{eq:result:large}).

Since exploring the behaviour of the heat kernels is very important in many areas of Mathematics (for example in differential equations, harmonic analysis, potential theory or theory of stochastic processes), the literature abounds with a huge amount of different kinds of deep and beautiful results on heat kernels estimates, which are difficult to briefly describe. However, even if we deal with the most classical case, i.e. the heat equation based on Laplacian in $\R^n$ (see \cite{Zhang:2002} or the works of E.B. Davies \cite{DaviesSimon:1984}, \cite{Davies:1987}, \cite{Davies:1990}, \cite{Davies:1991}) or we consider the corresponding problem on manifolds (see for example \cite{SC:2010} and the references within) or finally, we look at the resent results on the Fourier-Bessel heat kernel estimates on $(0,1)$ studied in \cite{NowakRoncal:2014a} and \cite{NowakRoncal:2014b}, the overwhelming majority of the results are qualitatively sharp, i.e. the constants appearing in the exponential terms in the upper and lower bounds are different. The presented result given in Theorem \ref{thm:main}, as well as the result of \cite{BogusMalecki:2014}, is free from this defect and here the exponential behaviour is described explicitly. However, this requirement imposes more complicated form of the part describing non-exponential behaviour of $p(t,x,y)$. We refer Reader to the Introduction in \cite{BogusMalecki:2014}, where more detailed discussion on this phenomenon was conducted.

The paper is organized as follows. The basic properties of the Bessel function $I_0(z)$ as well as the notation related to Bessel processes are recalled in Preliminaries. Moreover, in Lemma~\ref{lem:pt:derivative}, we introduce some technical estimates of a derivative of the heat kernel on $(0,\infty)$, which will be useful in the sequel. Then, in Section~\ref{sec:xyt:small}, we provide the proof of formula (\ref{eq:result:small}) which is divided into two parts. The first one relates to the case, when $x^2/t$ and $y^2/t$ are bounded (Proposition~\ref{prop:xy<t:y^2<t}) and in Proposition~\ref{prop:xy<t:y^2>t} we deal with the case when $(x\vee y)^2/t$ is large. Finally, in Section~\ref{sec:xyt:large} we justify (\ref{eq:result:large}).


\section{Preliminaries}


\label{section:preliminaries}
\subsection{Modified Bessel functions}
We begin with collecting some selected properties of the modified Bessel function $I_0(z)$ used in the sequel. We start with the following definition
\formula{
   I_0(z)=\sum_{n=0}^{\infty}\frac{1}{(n!)^2}\left(\frac{z}{2}\right)^{2n}\/,\quad z\in \R\/.
}
The above given series representation immediately implies that $I_0(z)$ is an increasing function on $[0,\infty)$. On the other side, the function $e^z I_0(z)$ is decreasing on $[0,\infty)$, i.e. we have (see Thm. 2.1 in \cite{Laforgia:1991}) 
\formula[eq:I:inequality]{
   \frac{I_0(y)}{I_0(x)}\leq e^{y-x}\/,\quad y>x>0\/.
}
Obviously $I_0(0)=1$ and the behavior at infinity is described by (see \cite{Erdelyi:1953:volII} 7.13.1 (5))
\formula[eq:I:asym:infty]{
  I_0(z) &\sim \frac{e^z}{\sqrt{2\pi z}}\left(1+O(1/z)\right)\/,\quad z\to\infty\/.
}
Finally, we recall the estimates of the ratio of $I_0(z)$ and its derivative $I_1(z)$ given in \cite{Nasell:1978}
\formula[eq:I:ratio:bound]{
  \frac{z}{z+2}\leq\frac{I_1(z)}{I_0(z)}\leq\frac{2z}{2z+1}, \quad z>0\/.
}

\subsection{Bessel processes}
We write $\pr_x^{(\mu)}$ and $\ex^{(\mu)}_x$ for the probability law and the corresponding expected value of the Bessel process $R_t^{(\mu)}$ with index $\mu\geq 0$ on the canonical path space with starting point $R_0^{(\mu)}=x>0$. The filtration of the coordinate process is denoted by $\mathcal{F}_t^{(\mu)}=\sigma\{R_s^{(\mu)}:s\leq t\}$. The transition probability density (with respect to the reference measure $m(dy)=ydy$) of two dimensional Bessel process ($\mu=0$) is given by
\formula[eq:transitiondensity:formula]{
p(t,x,y) = \frac{1}{t} 
\exp\left(-\frac{x^2+y^2}{2t}\right)I_{0}\left(\frac{xy}{t}\right)\/,\quad
x,y,t> 0\/.
}
The laws of Bessel processes with different indices are absolutely continuous. In particular, the corresponding Radon-Nikodym derivative for the laws of the Bessel processes with index $0$ and $1/2$ is described by
\formula[ac:formula]{
\left.\frac{d\pr^{(0)}_x}{d\pr^{(1/2)}_x}\right|_{\mathcal{F}_t^{(1/2)}}=\left(\frac{w(t)}{x}\right)^{-1/2}\exp\left(\frac{1}{8}\int_{0}^{t}\frac{ds}{w^{2}(s)}\right)\/,
}
where $x>0$, and the above given formula holds $\pr^{(1/2)}_x$-a.s..

Continuity of the paths and the Hunt formula enable us to write the corresponding transition density function for the process $R_t^{(0)}$ killed when it leaves a half-line $(a,\infty)$, $a>0$, in the following way
\formula[eq:hunt]{
  p_a(t,x,y) = p(t,x,y)-\int_0^t q_{x,a}(s)p(t-s,1,y)ds\/,\quad x,y>a\/,\quad t>0\/.
}
Here $q_{x,a}(s)$ denotes the density function of $T_a$, the first hitting time  of a level $a$ by the process starting from $x>a$. Its uniform sharp estimates were provided in \cite{BMR3:2013} (see Theorem $2$ therein). More precisely, we have
\formula[eq:q:behaviour:1]{
q_{x,1}(s)\approx \frac{x-1}{\sqrt{x}}\frac{1}{s^{3/2}}\exp\left({-\frac{(x-1)^2}{2s}}\right)\/,\quad s<2x\/,
}
and 
\formula[eq:q:behaviour:2]{
q_{x,1}(s)\approx\frac{x-1}{x}\frac{1+\ln x}{(1+\ln(s+x))(1+\ln(1+s/x))}\frac1s \exp\left({-\frac{(x-1)^2}{2s}}\right)\/, \quad s\geq2x\/.
}
These estimates lead to  the uniform bounds of the survival probabilities (see Theorem $10$ in \cite{BMR3:2013})
\formula[eq:q:survival:probability]{
   \int_{t}^{\infty}q_{x,1}(s)\,ds \approx 1 \wedge \frac{\ln x}{\ln (1+\sqrt{t})}\/, \quad x>1, t>0\/.
}
Note also that 
\formula[eq:p:integration]{
   \int_{1}^\infty p_1(t,x,y)\,m(dy) = \ex_x[T_1>t] = \int_{t}^\infty q_{x,1}(s)\,ds\/.
}

We end this section with the following very useful lemma, which refers to the uniform estimates of the derivative in $x$ of $p(t,x,y)$.




\begin{lemma}
\label{lem:pt:derivative}
We have
\formula{
\frac{\partial}{\partial x}p(t,x,y)\approx p(t,x,y)  \cdot x\left(\frac{y}{t}\right)^{2}\/,
}
whenever $xy\leq t $ and $y^2\geq 4t$.
\end{lemma}
\begin{proof}
Differentiating the right-hand side of (\ref{eq:transitiondensity:formula}) and using $I_0'(z)=I_1(z)$ lead to 
\formula{
\frac{\partial}{\partial x}p(t,x,y)=p(t,x,y)\cdot \frac{1}{t}\left(y\frac{I_1(xy/t)}{I_0(xy/t)}-x\right)\/,
}
for every $x,y>0$ and $t>0$. The upper bounds in \eqref{eq:I:ratio:bound} gives 
\formula{
y\frac{I_1(xy/t)}{I_0(xy/t)}-x\leq y\cdot \frac{2xy}{2xy+t}-x\leq \frac{2xy^2}{t}\/.
}
On the other-side, using the lower bounds in \eqref{eq:I:ratio:bound} and the conditions $xy\leq t$, $y^2\geq 4t$, we get
\formula{
y\frac{I_1(xy/t)}{I_0(xy/t)}-x\geq \frac{xy^2}{xy+2t}-x \geq \frac{xy^2}{3t}-x\geq \frac1{12} \cdot \frac{xy^2}{t}\/.
}
This ends the proof. 
\end{proof}

\begin{remark}
\label{rem:pt:derivative}
If we assume that $2x<y$ and $y^2\geq 4t$ (namely, we replace the condition $xy\leq t$ by $2x<y$) we can write, using lower bound given in \eqref{eq:I:ratio:bound}, that 
\formula{
y\frac{I_1(xy/t)}{I_0(xy/t)}-x\geq \frac{xy^2}{xy+2t}-x = \frac{x(y^2-xy-2t)}{xy+2t}> \frac{x(y^2/2-2t)}{xy+2t}\geq 0
}
and, as a consequence, we obtain that the function $x\mapsto p(t,x,y)$ is increasing for $y^2\geq 4t$ and $1<x<y/2$.
\end{remark}


\section{Estimates for $xy/t$ small}
\label{sec:xyt:small}
By symmetry, it is enough to consider $1<x<y$. Moreover, in this section we only deal with the case $xy<t$. It implies that $x^2\leq xy<t$ and consequently (\ref{eq:result:small}) reads as 
\formula{
  p_1(t,x,y)\approx\frac{\ln x \ln y}{\ln^2 3t}\,p(t,x,y)\/,\qquad y^2\leq t\/,
}
and 
\formula{
   p_1(t,x,y)\approx\frac{\ln x}{\ln (3t/y)}\,p(t,x,y)\/,\qquad y^2>t\/.
}
It naturally splits the proof into to parts related to the case when $y^2/t$ is small (Proposition \ref{prop:xy<t:y^2<t}) and when $y^2/t$ is large (Proposition \ref{prop:xy<t:y^2>t}).  

\begin{proposition}
\label{prop:xy<t:y^2<t} 
For every $m\geq 2$ we have
\formula{
  p_1(t,x,y)\stackrel{m}{\approx} \frac{\ln x \ln y}{\ln^2(3t)}\,p(t,x,y)\/,
}
whenever $xy<t$ and $y^2\leq mt$.
\end{proposition}
\begin{proof}
  Taking $x\to 0^+$ in (\ref{eq:I:inequality}) we get that $I_0(z)\leq e^z$ for every $z>0$, which together with the obvious general estimate $p_1(t,x,y)\leq p(t,x,y)$, see \eqref{eq:hunt}, gives
	\formula{
	  p_1(t,x,y)\leq \frac{1}{t}\exp\left(-\frac{x^2+y^2}{2t}\right)I_0(xy/t)\leq \frac{1}{t}\/,\quad x,y>1\/, t>0\/.
	}
	Applying this estimate to the middle term in the Chapman-Kolmogorov equation given below we arrive at
	\formula{
	   p_1(t,x,y) &= \int_1^\infty \int_1^\infty p_1(t/3,x,w)p_1(t/3,w,z)p_1(t/3,z,y)m(dw)m(dz)\\
		&\leq \frac{3}{t} \int_1^\infty p_1(t/3,x,w)\,m(dw) \int_1^\infty p_1(t/3,y,z)\,m(dz)\\
		&= \frac{3}{t} \int_{t/3}^\infty q_{x,1}(s)ds\,\int_{t/3}^\infty q_{y,1}(s)ds \stackrel{m}{\approx} \frac{\ln x}{\ln 3t}\frac{\ln y}{\ln 3t}\, p(t,x,y)\/, 
	}
	where the last line follows from the fact that $p(t,x,y)\stackrel{m}{\approx} 1/t$, whenever $xy\leq y^2\leq mt$, and (\ref{eq:q:survival:probability}) (note that $t>xy>1$ and consequently $\ln(1+\sqrt{t/3})\approx \ln3t$).
	
	To deal with the lower bounds we begin with showing that there exists a constant $c_m>0$ such that
	\formula[eq:low:midstep]
	{
	 p_1(t,x,w)\geq c_m \frac{\ln x}{\ln (3t)}\,p(t,x,w)
	}
	whenever $1<x<w/2$ and $w>1+2\sqrt{t}$. Indeed, denoting the subtrahend in \eqref{eq:hunt} as $r_1(t,x,y)$, we can write
\formula{
\frac{p_1(t,x,w)}{p(t,x,w)}=1-\frac{r_1(t,x,w)}{p(t,1,w)}\frac{p(t,1,w)}{p(t,x,w)},
}
where using (\ref{eq:I:inequality}) we have 
\formula{
\frac{r_1 (t,x,w)}{p(t,1,w)}&=\int_{0}^{t}q_{x,1}(s)\frac{t}{t-s}\,\exp{\left(-\frac{w^2 +1}{2}\left(\frac{1}{t-s}-\frac{1}{t}\right)\right)}\frac{I_{0}\left(w/(t-s)\right)}{I_{0}\left(w/t\right)}\,ds\\
&\leq\int_{0}^{t}q_{x,1}(s)\frac{t}{t-s}\exp{\left(-\frac{(w-1)^2}{2}\left(\frac{1}{t-s}-\frac{1}{t}\right)\right)\,ds}\/.
}
Consequently, using monotonicity of the function $u \mapsto u^{-1}\exp{\left(-{(w-1)^2}/{2u}\right)}$ for $w>1+2\sqrt{t}$ and $0<u<t$, we get
\formula{
\frac{r_1 (t,x,w)}{p(t,1,w)}\leq\int_{0}^{t}q_{x,1}(s)ds.
}
On the other hand, the monotonicity of $x\to p(t,x,w)$ showed in Remark \ref{rem:pt:derivative}
gives $p(t,1,w)\leq p(t,x,w)$. Hence, using \eqref{eq:q:survival:probability}, we obtain 
\formula{
\frac{p_1(t,x,w)}{p(t,x,w)}\geq 1-\int_{0}^{t}q_{x,1}(s)ds=\int_{t}^{\infty}q_{x,1}(s)ds\approx  \frac{\ln x}{\ln 3t}.
}
	Finally, the Chapman-Kolmogorov equation and (\ref{eq:low:midstep}) give
	\formula{
	   p_1(t,x,y) &\geq \int_{2\sqrt{mt}}^{3\sqrt{mt}}p_1(t/2,x,w)p_1(t/2,y,w)\,m(dw)\\
		&\geq c_m\frac{\ln x\ln y}{\ln^2 (3t)}  \int_{2\sqrt{mt}}^{3\sqrt{mt}} p(t/2,x,w)p(t/2,y,w)\,m(dw)
	}
	Since $I_0(z)\geq I_0(0)=1$, we get
	\formula{
	  p_1(t,x,y)&\geq  c_m \frac{\ln x\ln y}{\ln^2 (3t)} \frac{1}{t^2}\exp\left(-\frac{x^2+y^2}{t}\right)  \int_{2\sqrt{mt}}^{3\sqrt{mt}}  \exp\left(-\frac{2w^2}{t}\right)\,m(dw)\\
		&\stackrel{m}{\approx} \frac{\ln x\ln y}{\ln^2 (3t)}\frac{1}{t^2}\int_{2\sqrt{mt}}^{3\sqrt{mt}} m(dw) \stackrel{m}{\approx} \frac{\ln x\ln y}{\ln^2 (3t)}\,p(t,x,y)\/,
	}
	where the last estimate once again follows from $p(t,x,y)\stackrel{m}{\approx} 1/t$, which is true whenever $xy\leq y^2<mt$. The proof is complete.
\end{proof}
%

\begin{proposition}
\label{prop:xy<t:y^2>t}
  We have
\formula{
   p_1(t,x,y)\approx \frac{\ln x}{\ln(3t/y)}\,p(t,x,y) \/,
}
whenever $xy<t$ and $y^2\geq 16t$.
\end{proposition}
\begin{proof}
Note that for $y<52$ the conclusion follows from Proposition \ref{prop:xy<t:y^2<t} with $m=52^2$. Thus from now we will additionally assume that $y>52$. The second simplification comes from the fact that $p(t,x,y)\approx p(t,1,y)$ whenever $xy<t$. Taking that into consideration and using \eqref{eq:hunt}, we can write
\formula{
 \frac{p_1(t,x,y)}{p(t,x,y)}\approx \frac{p(t,x,y)-p(t,1,y)}{p(t,1,y)}+\int_t^\infty q_{x,1}(s)ds+\int_0^t \frac{p(t,1,y)-p(t-s,1,y)}{p(t,1,y)} q_{x,1}(s)\,ds\/.
}
To deal with the first part, we use the Mean-value theorem and Lemma \ref{lem:pt:derivative} to show that there exists $1<x_c<x$ such that for $xy<t$ and  $y^2>16t>4t$ we have
\formula{
p(t,x,y)-p(t,1,y)=(x-1)\cdot \left.\frac{\partial}{\partial x}p(t,x,y)\right|_{x=x_c}\approx  x_c(x-1)\frac{y^2}{t^2}p(t,x_c,y).
}
In particular, it implies that the difference is positive. Moreover, since $p(t,x_c,y)\approx p(t,1,y)$, there exists constant $c_1>0$ such that
\formula[eq:ptdifference]{
p(t,x,y)-p(t,1,y)\leq c_1\cdot x(x-1)\,\frac{y^2}{t^2}\, p(t,1,y).
}

The estimates of the second part are given in \eqref{eq:q:survival:probability}. Finally, we split the last integral into three parts
\formula{
\left[\int_0^{2x}+\int_{2x}^{\frac{4t^2}{y^2}}+\int_{\frac{4t^2}{y^2}}^{t}\right]\frac{p(t,1,y)-p(t-s,1,y)}{p(t,1,y)}q_{x,1}(s)ds = A_1+A_2+A_3
}
and estimate each of them separately by examine the behavior of the ration appearing under the integrals for the corresponding range of parameter of integration. More precisely, using the inequality $e^{-z}\geq 1-z$ and the fact that $I_0(z)\approx 1$ near zero we get the existence of the positive constant $c_2$ such that for every $s\leq 4t^2/y^2\leq t/4$ we have
\formula{
\frac{p(t,1,y)-p(t-s,1,y)}{p(t,1,y)} &=1-\frac{t}{t-s}\exp\left(-\frac{(1+y^2)s}{2t(t-s)}\right)\frac{I_0(y/(t-s))}{I_0(y/t)}\\
&\leq 1-\frac{t}{t-s}\left(1-\frac{(1+y^2)s}{2t(t-s)}\right)\frac{I_0(y/(t-s))}{I_0(y/t)}\\
& = 1-\frac{t}{t-s}\frac{I_0(y/(t-s))}{I_0(y/t)}+\frac{(1+y^2)s}{2(t-s)^2}\frac{I_0(y/(t-s))}{I_0(y/t)}\leq c_2\frac{y^2}{t^2}s\/.
}
Note that the difference of the first two components is negative and has been omitted in the last estimate. Since $s\leq 4t^2/y^2\leq t/4$, we have $\frac{y^2 s}{2t(t-s)}\leq \frac{y^2 \cdot (4t^2/y^2)}{2t(t-t/4)}= 8/3$. Thus, the inequality $e^{-z}\leq 1-z/3$, which is true for $0\leq z\leq 8/3$, gives 
\formula{
  \frac{p(t,1,y)-p(t-s,1,y)}{p(t,1,y)} &\geq 1-\frac{t}{t-s}\frac{I_0(y/(t-s))}{I_0(y/t)}+\frac{y^2s}{6t(t-s)}\frac{I_0(y/(t-s))}{I_0(y/t)}\/.
}
Observe that $\frac{ys}{t(t-s)}\leq \frac{y}{3t}\leq 1/3$ whenever $s\leq t/4$. Thus, 
using (\ref{eq:I:inequality}) and $e^{z}\leq 1+6z/5$ for $z\in(0,1/3)$, we obtain
\formula{
   \frac{t}{t-s}\frac{I_0(y/(t-s))}{I_0(y/t)}-1 &\leq \frac{t}{t-s}\left(\exp\left(\frac{ys}{t(t-s)}\right)-1\right)+\frac{s}{t-s}\\
	&\leq \frac{6ys}{5(t-s)^2}+\frac{s}{t-4t^2/y^2}\leq \frac{y^2}{t^2}s\cdot \left(\frac{32}{15y}+\frac{t}{y^2-4t}\right) \/.
	}
Moreover, the monotonicity of $I_0(z)$ gives
\formula{
  \frac{p(t,1,y)-p(t-s,1,y)}{p(t,1,y)} &\geq -\frac{y^2}{t^2}s\cdot \left(\frac{32}{15y}+\frac{t}{y^2-4t}\right)+\frac{y^2s}{6t(t-s)}\\
	&\geq\frac{y^2}{t^2}s\left(\frac{1}{6}-\frac{32}{15y}-\frac{t}{y^2-4t}\right)> \frac{s}{156} \frac{y^2}{t^2}\/,
}
since $y> 52$ and $y^2-4t\geq 12t$. Hence we have showed that
\formula{
\frac{p(t,1,y)-p(t-s,1,y)}{p(t,1,y)}\approx s \frac{y^2}{t^2}
}
whenever $s\leq 4t^2/y^2$. Thus, taking into consideration the estimates of $q_{x,1}(s)$ for $s\leq 2x$ (see \eqref{eq:q:behaviour:1}), we can write
\formula{
A_1 &\approx (x-1) \frac{y^2}{t^2}\int_0^{2x}\frac{1}{\sqrt{s}}\exp\left(-\frac{(x-1)^2}{2s}\right)\,ds
 =\frac{(x-1)^2}{\sqrt{2}}\frac{y^2}{t^2}\int_{\frac{(x-1)^2}{4x}}^{\infty}u^{-3/2}e^{-u}du\\
&=\frac{(x-1)^2}{\sqrt{2}}\frac{y^2}{t^2}\Gamma{\left(-1/2,\frac{(x-1)^2}{4x}\right)}.
}
where $\Gamma(\alpha,z)$ denotes the incomplete Gamma function. Since $\Gamma(a,z)\approx z^a$ as $z\to 0^+$ and $\Gamma{(a,z)}\approx z^{a-1}\exp{(-z)}$ if $z\rightarrow \infty$ we obtain the following  
\formula[eq:A1]{
   A_1\approx \frac{y^2}{t^2}\left((x-1)\wedge \sqrt{x}\exp\left(-\frac{(x-1)^2}{4x}\right)\right)\/.
}
On the other side, the estimates of $q_{x,1}(s)$ for $s\geq 2x$ \eqref{eq:q:behaviour:2} gives
\formula[eq:A2:midstep]{
A_2\approx \ln x \cdot \frac{y^2}{t^2}\cdot \int_{2x}^{4t^2/y^2}\frac{1}{\ln(s)\ln(s/x)}\exp{\left(-\frac{(x-1)^2}{2s}\right)}ds
}
and obviously
\formula{
  \int_{2t^2/y^2}^{4t^2/y^2}\frac{1}{\ln s\ln(s/x)}\exp{\left(-\frac{(x-1)^2}{2s}\right)}ds\geq \frac{2t^2}{y^2}\frac{\exp\left(-\frac{(x-1)^2y^2}{4t^2}\right)}{\ln (4t^2/y^2) \ln (4t^2/(xy^2))}\geq \frac{e^{-1/4}}{2}\frac{t^2}{y^2\ln^2(3t/y)}\/.
}
In a similar way, dropping the exponential term in the integral in (\ref{eq:A2:midstep}) and estimating the logarithms by its values at $2t^2/y^2$ lead to 
\formula{
\int_{2t^2/y^2}^{4t^2/y^2}\frac{ds}{\ln(s)\ln(s/x)} \leq 2\int_{2t^2/y^2}^{4t^2/y^2}\frac{ds}{\ln^2s} \leq \frac{4t^2}{y^2}\frac{1}{\ln^2(2t^2/y^2)}\approx \frac{t^2}{y^2}\frac{1}{\ln(3t/y)}\/,
}
where the first inequality follows from $s^2/x^2\geq 2t^2/(x^2y^2)s\geq s$. Collecting all together we have just obtained that
\formula[eq:A2]{
A_2\approx\frac{\ln x}{\ln^2 (3t/y)}\/.
}
Finally, to deal with $A_3$, observe that (\ref{eq:I:inequality}) gives
\formula[eq:A3:midstep]{
1\geq \frac{p(t,1,y)-p(t-s,1,y)}{p(t,1,y)} &\geq 1-\frac{t}{t-s}\exp\left(-\frac{(y-1)^2}{2t}\frac{s}{t-s}\right)\/.
}
If $s\in(4t^2/y^2,t/2)$ then $t/(t-s)\leq 2$ and $\frac{(y-1)^2}{2t}\frac{s}{t-s}\geq \frac{y^2}{4t}\frac{s}{t-s}\geq 1$ (since $y\geq 52$). Consequently, the right-hand side of (\ref{eq:A3:midstep}) is bounded from below by $1-2e^{-1}>0$. On the other hand, since $\frac{s}{t-s}\geq \frac{t}{2(t-s)}$ for $s\geq t/2$ and the function 
\formula{
  s \mapsto \frac{t}{t-s}\exp\left(-\frac{(y-1)^2}{4t}\frac{t}{t-s}\right)
}
is decreasing on $(t/2,t)$, we get the lower-bounds of the form
\formula{
  \frac{p(t,1,y)-p(t-s,1,y)}{p(t,1,y)} &\geq 1-2\exp\left(-\frac{(y-1)^2}{2t}\right)\geq 1-2e^{-1}>0\/,\quad s\in (t/2,t)\/.
}
This together with the estimates of $q_{x,1}(s)$ for large $s$ give
\formula{
  A_3 \approx \int_{4t^2/y^2}^{t}q_x(s)ds \approx
\ln x \int_{4t^2/y^2}^{t}\frac{1}{s}\frac{1}{\ln(s+x)\ln(1+s/x)}\exp{\left(-\frac{(x-1)^2}{2s}\right)}ds\/.
}
and since $2s\geq s+x\geq s$, $1+s/x\geq s/x\geq \sqrt{s}$ and   $(x-1)^2/(2s)\leq(x-1)^2 t^2/(8y^2)\leq 1/8$ on $(4t^2/y^2, t)$ we have
\formula[eq:A3]{
A_3 \approx \ln x \int_{4t^2/y^2}^{t}\frac{ds}{s\ln^2 s}=\ln x\left(\frac{1}{\ln(4t^2/y^2)}-\frac{1}{\ln (2t)}\right)\/.
}
Now, by compare \eqref{eq:A1} and \eqref{eq:A2}, we can observe that $A_2$ dominates $A_1$.

Hence combining together \eqref{eq:q:survival:probability}, \eqref{eq:A2} and \eqref{eq:A3} we obtain
  \formula{
\int_t^\infty q_{x,1}(s)ds +   A_1+A_2+A_3\approx  \ln x\left(\frac{1}{\ln (2t)}+\frac{1}{\ln^2(3t/y)}+ \frac{1}{\ln(3t/y)}-\frac1{\ln (2t)}\right)\approx \frac{\ln x}{\ln(3t/y)}\/,
}
where the last estimates holds since $xy/t\leq 1$ and consequently $1/\ln(3t/y)$ dominates $1/\ln^2(3t/y)$. Finally, since (see \eqref{eq:ptdifference})
\formula{
  0<\frac{p(t,x,y)-p(t,1,y)}{p(t,1,y)}\leq c_1 \cdot x(x-1)\frac{y^2}{t^2}\/,
}
this part is also dominated by $\ln x/\ln(3t/y)$ and the proof is complete.
\end{proof}


\section{Estimates for $xy/t$ large}
\label{sec:xyt:large}

In this Section we focus on proving (\ref{eq:result:large}). The main ideas of the proofs in this case come from \cite{BogusMalecki:2014}, but for convenience of the Reader we adapt the arguments to the case $\mu=0$ and present them in condensed form in the following two proposition related to the lower and upper bounds respectively. Similarly as in the previous section, we assume that $y>x>1$, but now we consider $xy\geq t$.

\begin{proposition}\label{prop:lowerbounds:xytlarge}
There exist constant $C_1>0$ such that for every $x,y>1$ and $t>0$ we have
   \formula[eq:up:low]{
      p_1(t,x,y)\geq C_1 \left(1\wedge \frac{(x-1)(y-1)}{t}\right)\frac{1}{\sqrt{xyt}}\exp\left(-\frac{(x-y)^2}{2t}\right),}
whenever $xy\geq t.$
\end{proposition}
\begin{proof}
For any Borel set $A\subset(1,\infty)$, using the absolute continuity property (\ref{ac:formula}) and simply estimating the exponential term by $1$, we obtain
\formula{
\int_{A}p_{1}(t,x,y)m(dy) &=\ex_{x}^{(1/2)}\left[T_1^{(1/2)}>t,R_t \in A;
\left(\frac{R_t}{x}\right)^{-1/2}\exp\left(\frac{1}{8}\int_{0}^{t}\frac{ds}{R^{2}_s}\right)\right]\\
&\geq x^{1/2}\cdot\ex_{x}^{(1/2)}[T_1^{(1/2)}>t,R_t \in
A;(R_t)^{-1/2}]\\
&= x^{1/2}\cdot\int_A y^{-1/2}\cdot p_1^{(1/2)}(t,x,y)\,m^{(1/2)}(dy)\/.}
Since $m^{(1/2)}(dy)= y\cdot m(dy)$ it follows that
\formula[eq:munu:relation]{
p_{1}(t,x,y)\geq&(xy)^{1/2}p_{1}^{(1/2)}(t,x,y)\/.
}
Thus, we conclude by using the explicit formula for $p_{1}^{(1/2)}(t,x,y)$, which is just the Dirichlet kernel for the classical Laplacian multiplied by the factor $(xy)^{-1}$ (see for example \cite{BorodinSalminen:2002} for more details)
\formula{
p_1^{(1/2)}(t,x,y) &= \frac{1}{\sqrt{2\pi t}}\frac{1}{xy} \left(\exp{\left(-\frac{(x-y)^2}{2t}\right)}-\exp{\left(-\frac{(x+y)^2}{2t}\right)}\right)\\
&\approx\left(1\wedge\frac{(x-1)(y-1)}{t}\right)\frac{1}{xy}\frac{1}{\sqrt{t}}\exp{\left(-\frac{(x-y)^2}{2t}\right)}\/.
}
\end{proof}

\begin{proposition}
There exist constant $C_2>0$ such that for every $x,y>1$ and $t>0$ we have
\formula{
p_1(t,x,y)\leq C_2 \left(1\wedge \frac{(x-1)(y-1)}{t}\right)\frac{1}{\sqrt{xyt}}\exp\left(-\frac{(x-y)^2}{2t}\right),
}
whenever $xy\geq t.$
\end{proposition}
\begin{proof}
Notice that if the space variables are bounded away from the boundary, i.e. $y>x>2$, then we have
\formula{
  \frac{(x-1)(y-1)}{t}\geq\frac{xy}{4t}>\frac{1}{4}
}
and consequently the result follows from $p_1(t,x,y)\leq p(t,x,y)$ and the asymptotic expansion of $I_0(z)$ for large arguments (see \eqref{eq:I:asym:infty}), which implies the estimates of $p(t,x,y)$ of the desired form for $xy>t$. Moreover, the result for small $t$, i.e. $t<4$, follows in a similar way as in Proposition \ref{prop:lowerbounds:xytlarge}. Indeed, since
\formula{
 \frac{1}{8} \int_0^{t}\frac{ds}{R^2_s}\leq\int_0^{t}\frac{ds}{8}\leq \frac{1}{2}
}
on $\{T_1^{(1/2)}>t\}$, using the absolute continuity property (\ref{ac:formula}), we obtain
\formula{
p_{1}(t,x,y) &\leq e^{1/2}\,(xy)^{1/2}p_{1}^{(1/2)}(t,x,y)
}
and the estimates of $p_{1}^{(1/2)}(t,x,y)$ gives the result.

Thus, it is enough to consider $1<x<2$ and $t\geq 4$. Notice that since $xy>t$, we have $y>4/x>2$. Moreover, for $0<s<t$, the following inequalities hold
\formula
{
I_{0}\left(\frac{y}{t-s}\right)\geq I_{0}\left(\frac{y}{t}\right) ,\quad \frac{1}{t-s}\geq\frac{1}{\sqrt{t}\sqrt{t-s}},
}
and consequently, we can write
\formula{
r_1(t,x,y)=\int_0^{t}q_{x,1}(s)p(t-s,1,y)ds\geq\frac{I_0(y/t)}{\sqrt{t}}\int_0^{t}\frac{1}{\sqrt{t-s}}\exp\left({-\frac{y^2 +1}{2(t-s)}}\right)q_{x,1}(s)ds\/.
}
Now using the following lower bounds for $q_{x,1}$ (see Lemma 4 in \cite{BMR3:2013})
\formula{
q_{x,1}(s)\geq \frac{x-1}{\sqrt{2\pi x}}\frac{1}{s^{3/2}}\exp\left(-{\frac{(x-1)^2}{2s}}\right), \quad s>0\/,
}
we arrive at
\formula{
r_1(t,x,y)\geq\frac{x-1}{\sqrt{2\pi x}}\frac{I_0(y/t)}{\sqrt{t}}\int_0^{t}\frac{1}{\sqrt{t-s}}\frac{1}{s^{3/2}}\exp\left(-{\frac{(x-1)^2}{2s}}\right)\exp\left({-\frac{y^2 +1}{2(t-s)}}\right)ds.
}
Substituting $w=1/s-1/t$ in the integral above, we can reduce it to 
\formula{
t^{-1/2}\exp\left(-\frac{(x-1)^2+y^2+1}{2t}\right)\int_0^{\infty}w^{-1/2}\exp\left(-w\frac{(x-1)^2}{2}\right)\exp\left(-\frac{y^2+1}{2t^2}\frac{1}{w}\right)dw.
}
One can recognize that the above-given integral appears in the integral representation of the modified Bessel function $K_{1/2}(z)=\sqrt{\pi/(2z)}e^{-z}$ and consequently, it can be computed explicitly (see for example formula 3.471.15 in \cite{GradsteinRyzhik:2007}). Finally, we get the estimates of the $r_1(t,x,y)$ in the following form
\formula[r1:bound:midstep]{
r_1(t,x,y)\geq\frac{1}{\sqrt{x}}\frac{I_0(y/t)}{t}\exp\left(-\frac{(x-1+\sqrt{y^2+1})^2}{2t}\right).
}
On the other hand, because of \eqref{eq:I:inequality}, we have
\formula[p:bound:midstep]{
p(t,x,y)\leq\frac{I_0(y/t)}{t}\exp\left(-\frac{(x-y)^2}{2t}\right)\exp\left(-\frac{y}{t}\right)
}
and consequently, from \eqref{r1:bound:midstep} and \eqref{p:bound:midstep}, we get
\formula{
p_1(t,x,y)
&\leq\frac{I_0(y/t)}{t}\exp\left(-\frac{(x-y)^2}{2t}\right)\exp\left(-\frac{y}{t}\right)
\left[1-\frac{1}{\sqrt{x}}\exp\left(-\frac{(x-1)(\sqrt{y^2+1}+y-1)}{2t}\right)\right]
}
Now observe that
\formula{
1-\frac{1}{\sqrt{x}}\exp\left(-\frac{(x-1)(\sqrt{y^2+1}+y-1)}{2t}\right)\leq\frac{x-1}{\sqrt{x}}\exp\left(-\frac{(x-1)(\sqrt{y^2+1}+y-1)}{2t}\right)\/.
}
Estimating the exponential term above by $1$ and taking into account that for $1<x<2\leq y$ and $xy\geq t$ we have
\formula{
\frac{x-1}{\sqrt{x}}=\frac{(x-1)(y-1)}{t}\frac{t}{\sqrt{x}\cdot(y-1)}\leq\frac{2}{\sqrt{x}}\cdot\frac{(x-1)(y-1)}{t}\/,
}
together with the obvious estimate
\formula{
1-\frac{1}{\sqrt{x}}\exp\left(-\frac{(x-1)(\sqrt{y^2+1}+y-1)}{2t}\right)\leq 1\/,
}
we arrive at
\formula{
p_1(t,x,y)\leq \frac{2}{\sqrt{x}}\left(1\wedge\frac{(x-1)(y-1)}{t}\right)\frac{I_0(y/t)}{t}\exp\left(-\frac{(x-y)^2}{2t}\right)\exp\left(-\frac{y}{t}\right).
}
Finally, the asymptotic description of $I_0(z)$ at infinity \eqref{eq:I:asym:infty} ensures that there exists constant $c_1>0$ such that $I_0(y/t)\leq c_1 \sqrt{t/y}\cdot e^{y/t}$ and consequently
\formula{
p_1(t,x,y)\leq 2 c_1 \cdot\frac{1}{\sqrt{xyt}}\left(1\wedge\frac{(x-1)(y-1)}{t}\right)I_0\left(\frac{xy}{t}\right)\exp\left(-\frac{(x-y)^2}{2t}\right)\/,
}
where we have additionally used the fact that $1\leq I_0(xy/t)$. Hence the proof is complete.
\end{proof}
\bibliographystyle{plain}
\bibliography{bibliography}

\end{document}